\documentclass[reqno,12pt]{amsart}
\usepackage{amsfonts,amssymb}
\usepackage{enumerate}
\usepackage{verbatim}

\newtheorem*{thm}{Main Theorem}
\newtheorem{lem}{Lemma}
\newtheorem*{prop}{Proposition}
\theoremstyle{definition}
\theoremstyle{remark}

\newtheorem{conj}{Conjecture}

\author{Jing-Jing Huang}
\address{
Jing-Jing Huang: Department of Mathematics, University of Toronto,
40 St. George Street, Toronto, Ontario, Canada M5S 2E4}
\email{huang@math.toronto.edu}
\thanks{Research is partially supported by the NSERC Grant A5123.}
 \dedicatory{Dedicated to Professor Wen-Ching Winnie Li on the occasion of her birthday}

\begin{document}

\title
{Hausdorff theory of dual approximation on planar curves}

\begin{abstract}
Ten years ago, Beresnevich-Dickinson-Velani \cite{BDV1} initiated a project that develops the general Hausdorff measure theory of dual approximation on non-degenerate manifolds. In particular, they established the divergence part of the theory based on their general ubiquity framework. However, the convergence counterpart of the project remains wide open and represents a major challenging question in the subject. Until recently, it was not even known for any single non-degenerate manifold. In this paper, we settle this problem for all curves in $\mathbb{R}^2$, which represents the first complete theory of its kind for a general class of manifolds.
\end{abstract}
\maketitle

\section{Introduction} \label{s1}

\subsection{Motivation and history}
Let $\mathbf{q}=(q_1,q_2,\cdots,q_n)\in\mathbb{Z}^n$, $\mathbf{x}=(x_1,x_2,\cdots,x_n)\in\mathbb{R}^n$, $\|\cdot\|$ be the distance to the nearest integer and $|\cdot|$ be the $L^\infty$ norm of vectors in $\mathbb{R}^n$. 
In the theory of Diophantine approximation, one considers the inequality
\begin{equation}\label{e1}
\|\mathbf{q}\cdot\mathbf{x}\|<\frac1{|\mathbf{q}|^\nu},
\end{equation}
where $\mathbf{q}\cdot\mathbf{x}=q_1x_1+\cdots+q_nx_n$ is the standard inner product on $\mathbb{R}^n$. In contrast to the \emph{simultaneous approximation}, traditionally this setting \eqref{e1} is referred as the \emph{dual approximation} (see \cite{BD} for more basics).

Dirichlet's theorem asserts that if $\nu\le n$, then for all $x\in\mathbb{R}^n$, the inequality \eqref{e1} admits infinitely many solutions in $\mathbf{q}$. The point $\mathbf{x}$ is called \emph{very well approximable} (abbr. VWA) if for some $\nu>n$, \eqref{e1} holds for infinitely many $\mathbf{q}\in\mathbb{Z}^n$. Thus, for VWA points, the exponent in \eqref{e1} can be improved beyond Dirichlet. Nevertheless, a straightforward application of the Borel-Cantelli lemma reveals that almost all $\mathbf{x}\in\mathbb{R}^n$ are not VWA. Thus the Dirichlet exponent $\nu=n$ is a natural threshold for \eqref{e1} from the metrical point of view.

There are naturally two manners of generalization. One may ask what is the precise threshold for \eqref{e1} to have infinitely many solutions when restricting $\mathbf{x}$ to lying on a proper submanifold of $\mathbb{R}^n$ and/or replacing the right side of \eqref{e1} with a general approximation function $\psi(|\mathbf{q}|)$. This turns out to be the key question that motivates many modern developments of the theory, on which we will do a brief survey.

The theory of Diophantine approximation on manifolds dates back to 1930s when Mahler \cite{Ma} raised his conjecture in transcendence theory, which states that almost all points on the Veronese curve
$$\mathcal{V}_n:=\{(x,x^2,\cdots,x^n)|x\in\mathbb{R}\}$$
are not VWA. This rather difficult question was eventually solved by Sprind\v{z}uk \cite{Sp1} in 1964. After this, Sprind\v{z}uk proposed an important general conjecture.
\begin{conj}[Sprind\v{z}uk]
Any analytic non-degenerate submanifold of $\mathbb{R}^n$ is extremal.
\end{conj}
A manifold $\mathcal{M}$ is \emph{extremal} if almost all points of $\mathcal{M}$ (with respect to the natural induced measure on $\mathcal{M})$ are not VWA. A manifold $\mathcal{M}$ is \emph{non-degenerate} if at almost all points $\mathbf{x}\in\mathcal{M}$, $\mathcal{M}$ has at most finite order of contact with any hyperplane that passes through $\mathbf{x}$. It is easily seen that an analytic manifold is non-degenerate if and only if, around almost all points, the local chart functions together with 1 are linearly independent over $\mathbb{R}$. Essentially, non-degeneracy guarantees that the manifold is sufficiently curved so as to deviate from any hyperplanes.

The philosophy underlying this conjecture can be summarized as that, in the sense of Diophantine approximation, a generic point on a generic submanifold of $\mathbb{R}^n$ should behave like a generic point in $\mathbb{R}^n$. In other words, as far as certain Diophantine approximation properties are concerned, one cannot distinguish a generic point on a non-degenerate submanifold of $\mathbb{R}^n$ from a generic point in $\mathbb{R}^n$. We will see that this very same idea arises in most of the major (solved or unsolved) questions in the area of metric Diophantine approximation on manifolds.

The fundamental Sprind\v{z}uk conjecture was proved by Kleinbock and Margulis \cite{KM} in 1998 using dynamical tools based on unipotent flows on homogeneous spaces. Their proof has served as a catalyst for many more recent developments (They actually proved the stronger Baker-Sprind\v{z}uk conjecture). Among other things, the Khintchine-Groshev theory is a natural generalization of the theory of extremal manifolds obtained by replacing the right side of \eqref{e1} with a monotonic function $\psi$ of $|\mathbf{q}|$. 

From now on, we will call the decreasing function $\psi:\mathbb{N}\rightarrow\mathbb{R}^+$ an \emph{approximation function}. For a given approximation function $\psi$, let

\begin{equation}\label{e2}
A_n(\psi):=\{\mathbf{x}\in\mathbb{R}^n:\exists^\infty \mathbf{q}\in\mathbb{Z}^n\textrm{ such that } \|\mathbf{q}\cdot\mathbf{x}\|<\psi(|\mathbf{q}|)\}.
\end{equation}
Here $\exists^\infty$ means ``there exists infinitely many ...". 

Let $|\cdot|_\mathcal{M}$ denote the induced Lebesgue measure on $\mathcal{M}$. One wants to study the size of $A_n(\psi)\cap\mathcal{M}$ under $|\cdot|_\mathcal{M}$. Indeed, for $\mathcal{M}=\mathcal{V}_n$, A. Baker \cite{Bak} conjectured that

\begin{conj}[A. Baker]\label{c2}
\begin{equation*}
|A_n(\psi)\cap\mathcal{M}|_\mathcal{M}=\left\{
\begin{array}{lll}
\textnormal{\small{Z\tiny{ERO}}}&\textnormal{if}& \displaystyle\sum_{q=1}^{\infty}q^{n-1}\psi(q)<\infty\\
\textnormal{\small{F\tiny{ULL}}}&\textnormal{if}&
\displaystyle\sum_{q=1}^{\infty}q^{n-1}\psi(q)=\infty.
\end{array}
\right.
\end{equation*}
\end{conj}

Here $|\cdot|_\mathcal{M}=\textnormal{\small{F\tiny{ULL}}}$ means that the complement on $\mathcal{M}$ has measure zero. The highly nontrivial special case $\mathcal{V}_n$ was finally solved in \cite{Be2,Ber}. More importantly, it has been subsequently shown that the conjecture \ref{c2} holds for all non-degenerate submanifolds of $\mathbb{R}^n$ \cite{BKM,BBKM,Be1}. So the Lebesgue measure theory for Groshev type approximation on manifolds is complete. 

Beyond this, one may develop the much deeper Hausdorff theory in which one replaces the coarse Lebesgue measure $|\cdot|_{\mathcal{M}}$ with finer Hausdorff measures $\mathcal{H}^s$. We will recall the definitions of Hausdorff measure and dimension in \S{\ref{s2}}. As a byproduct, this approach yields the Hausdorff dimension of the set in question. This type of investigation seems to be initiated by A. Baker and W.M. Schmidt \cite{BS} in 1970, who proved that

\begin{equation}\label{e4}
\frac{n+1}{\nu+1}\le\dim(A_n(q^{-\nu})\cap\mathcal{V}_n)\le \frac{2(n+1)}{\nu+1}\quad\text{for any }\nu>n.
\end{equation} 

In the same paper, they also conjectured that the left inequality in \eqref{e4} should be equality, which was eventually solved by Bernik \cite{Ber1} in 1983. In 2000, Dickinson and Dodson \cite{DD} proved that for any extremal submanifold $\mathcal{M}$ of $\mathbb{R}^n$ one has the lower bound

 \begin{equation}\label{e5}
\dim(A_n(q^{-\nu})\cap\mathcal{M})\ge \frac{n+1}{\nu+1}+\dim\mathcal{M}-1\quad\text{for any }\nu>n.
\end{equation} 

Proving that the lower bound in \eqref{e5} is actually sharp represents a major open question, which is called the Generalized Baker-Schmidt Problem for Hausdorff dimension in \cite{BBV}.

\begin{conj}[GBSP for Hausdorff dimension]\label{c3}
For any submanifold $\mathcal{M}$ of $\mathbb{R}^n$ which is non-degenerate everywhere except possibly on a set of Hausdorff dimension $\le D$,  one has
\begin{equation*}
\dim(A_n(q^{-\nu})\cap\mathcal{M})=D\quad\text{for any }\nu>n
\end{equation*}
where $$D:=\frac{n+1}{\nu+1}+\dim\mathcal{M}-1.$$
\end{conj}

Besides the Veronese curve $\mathcal{V}_n$, Conjecture \ref{c3} was also known for non-degenerate planar curves due to R.C. Baker \cite{Ba}. By setting $\psi(q)=q^{-\nu}$, we immediately see that Conjecture \ref{c3} is an easy corollary of the following very precise and delicate conjecture, which is referred as the Generalized Baker-Schmidt Problem for Hausdorff measure in \cite{BBV}.

\begin{conj}[GBSP for Hausdorff measure]\label{c4}
For any approximation function $\psi$, any $s>m-1$ and any submanifold $\mathcal{M}$ of $\mathbb{R}^n$ of dimension $m$ which is non-degenerate everywhere except possibly on a set of zero Hausdorff $s$-measure\footnote{In their original formulation of GBSP for Hausdorff dimension and measure, the authors of \cite{BBV} seem to have overlooked this absolutely necessary condition.}, one has
\begin{equation*}
\mathcal{H}^s(A_n(\psi)\cap\mathcal{M})=\left\{
\begin{array}{lll}
0&\textnormal{if}& \displaystyle\sum_{q=1}^{\infty}\left(\frac{\psi(q)}{q}\right)^{s+1-m}q^n<\infty,\\
\mathcal{H}^s(\mathcal{M})&\textnormal{if}&
\displaystyle\sum_{q=1}^{\infty}\left(\frac{\psi(q)}{q}\right)^{s+1-m}q^n=\infty.
\end{array}
\right.
\end{equation*}
\end{conj}

Clearly, the case $s=m$ reduces to the Lebesgue measure theory and hence was known. The case $s>m$ is simply trivial. When $s<m$, the divergence case of Conjecture \ref{c4} has been established by Beresnevich, Dickinson and Velani \cite{BDV1} using their general ubiquity framework. However, the convergence case is much more difficult and represents a challenging question. Indeed, it is not known for any non-degenerate manifolds except for the parabola which was solved very recently \cite{Hus}. In this paper, we will settle this conjecture for all non-degenerate planar curves. This is the first result of its kind for a general class of manifolds.

Before stating our main theorem, we should also remark that there is a completely parallel  line of developments in the context of simultaneous approximation on manifolds, in which one considers the set
\begin{equation}\label{e30}
S_n(\psi):=\{\mathbf{x}\in\mathbb{R}^n:\exists^\infty {q}\in\mathbb{Z}\textrm{ such that } \max_{1\le i\le n}\|qx_i\|<\psi(q)\},
\end{equation} 
instead of the set $A_n(\psi)$, and the measure-theoretic property of the intersection $S_n(\psi)\cap\mathcal{M}$. One can then formulate a complete analogue of Conjecture \ref{c4}, namely the Hausdorff theory for simultaneous approximation on manifolds (see \cite[Conjecture 8.2]{Be3} for the precise statement). This analogue of Conjecture \ref{c4} for all non-degenerate planar curves has been remarkably established by Vaughan and Velani \cite{VV} for the convergence case, and Beresnevich, Dickinson and Velani \cite{BDV2} for the divergence case. Moreover, the divergence case of this analogous conjecture for all non-degenerate analytic manifolds has been subsequently solved by Beresnevich \cite{Be3}. Despite the superficial similarity between the simultaneous and dual approximations on manifolds, the two problems exhibit quite differing natures and difficulties; hence the solution of one of them unfortunately does not yield that of the other one by any reasonable means.  Our main theorem below is a natural counterpart of the result of Vaughan and Velani \cite{VV} in the dual setting, and hence brings the development of the dual approximation on manifolds in line with that of the simultaneous approximation on manifolds. 

\subsection{The main result}

\begin{thm}\label{t0}
Conjecture \ref{c4} holds when $\mathcal{M}$ is a curve in $\mathbb{R}^2$. Namely,
for any approximation function $\psi$, $s\in(0,1]$ and any $C^2$ planar curve $\mathcal{C}$ which is non-degenerate everywhere except possibly on a set of zero Hausdorff $s$-measure, we have
\begin{equation*}
\mathcal{H}^s(A_2(\psi)\cap\mathcal{C})=\left\{
\begin{array}{lll}
0&\textnormal{if}& \displaystyle\sum_{q=1}^{\infty}\left(\frac{\psi(q)}{q}\right)^{s}q^2<\infty,\\
\mathcal{H}^s(\mathcal{C})&\textnormal{if}&
\displaystyle\sum_{q=1}^{\infty}\left(\frac{\psi(q)}{q}\right)^{s}q^2=\infty.
\end{array}
\right.
\end{equation*}
\end{thm}

As remarked above, the divergence case in the Main Theorem is established in \cite{BDV1}. So the convergence case is the new substance, which we will prove in \S{\ref{s3}}. 

One can replace the Hausdorff $s$-measure $\mathcal{H}^s$ in the Main Theorem with a Hausdorff $g$-measure $\mathcal{H}^g$ for a general dimension function $g$, assuming the following growth condition\footnote{Incidentally, this very same condition was also assumed in \cite{Hu}.} on $g$ 
$$x^{s_1}\ll g(x)\ll x^{s_2}$$
with $2s_1<3s_2$. However, this reveals no additional information on the problem and will inevitably obscure the main thrust of the paper. We are happy to leave the details to the reader.

\subsection{Further general discussions}

Our approach in this paper is partially motivated by the work of Schmidt \cite{Sc1, Sc2}, who showed in 1964 that every $C^3$ non-degenerate curve is extremal.  This result of Schmidt, together with Sprind\v{z}uk's solution of Mahler's conjecture \cite{Sp1}, undoubtedly represent two pioneering work in the subject. In spite of this, the two methods they introduced differ quite significantly in nature. Sprind\v{z}uk invented the method which is now called \emph{the method of essential and inessential domains} (see \cite{Sp1, Sp2}); while Schmidt's method is based on an interesting arithmetic counting problem \cite{Sc1} which is now understood to be closely related to the problem of counting rational points lying near planar curves. As discussed above, the Hausdorff theory is much deeper than the extremal theory and hence requires very strong arithmetic input, which Schmidt's result \cite{Sc1} unfortunately fails to deliver.  The novel feature of this paper is that we adapt a counting result of the same type due to Huxley \cite{Hu}, which is essentially best possible, into an analytic form that is particularly 
suited for the application in Schmidt's method. This connection is established through the use of the \emph{dual curve}, which appears, in one form or another, in the previous works \cite{Bad,BDV2,BZ,Hua,Hu,VV}.

Another interesting feature to be seen, on checking our argument carefully, is that we do not need the full power of Huxley \cite{Hu}. Actually, we still obtain results of the same quality when the upper bound in Lemma \ref{l3} is relaxed to be $\ll\delta^{\alpha}R^\beta+R^{3/2-\varepsilon}$ for some $\varepsilon>0$ and $\alpha,\beta$ satisfying $\beta<\frac32+2\alpha$. In contrast, we recall that  when the corresponding Jarn\'{i}k type theorem for simultaneous approximation on planar curves was established in \cite{VV} one had to use Huxley in the strongest form. This discrepancy is probably not surprising since it is generally believed that dual approximation on manifolds is not more difficult than simultaneous approximation on manifolds. Nevertheless, it is indeed a bit surprising that the dual case for planar curves remained open for almost a decade after the simultaneous case was solved completely \cite{BDV2,VV}! 

We emphasize that the relation between the distribution of rational points near manifolds and simultaneous approximation on manifolds is widely known and well understood \cite{Be3,BDV2,BZ,Hua,VV}. However, to the best of our knowledge, the explicit relation between this arithmetic counting problem and dual approximation on manifolds seems to have never appeared in the literature; hence it is the purpose of this paper to address this issue. Beyond Khintchine's transference principle, little has been known about the relation between simultaneous and dual approximations.

We remark in passing that GBSP has been extended to the settings of inhomogeneous approximation and/or multi-variable approximation function. See \cite{Bad,BBV,BV2} and references therein. The inhomogeneous and multiplicative theory have also been developed in the simultaneous case \cite{BVV,BV1}. It is likely that the method in this paper can be adapted in order to cope with the above variations of the original problem. We hope to return to this in a future publication. We note however that it is straightforward to see that one may insert an inhomogeneous constant term $\theta$ in our problem and the whole argument still works.

\section{Hausdorff measure and dimension}\label{s2}
Here for completeness we recall the the definitions of Hausdorff measure and dimension. For more examples and discussions, one is referred to the book of Falconer \cite{Fa} and references therein. Let $X$ be a subset of $\mathbb{R}^n$ and $\{B_i\}$ be a collection of balls in $\mathbb{R}^n$ with diameters $\text{diam}(B_i)<\rho$ for a fixed $\rho>0$ such that $X\subseteq\bigcup_i B_i$, which is called a \emph{$\rho$-cover of $X$}.

 Now define
$$
\mathcal{H}_\rho^{s}(X):=\inf\left\{\sum_{i}(\text{diam}(B_i))^s\right\}
$$
where the infimum is taken over all possible $\rho$-covers of $X$. The Hausdorff $s$-measure $\mathcal{H}^s$ is defined as
$$
\mathcal{H}^s(X):=\lim_{\rho\rightarrow0}\mathcal{H}_\rho^s(X).
$$ 
Due to monotonicity, this limit either converges to some finite number or approaches $\infty$.

When $s$ is a positive integer, $\mathcal{H}^s$ is equivalent to the usual $s$-dimensional Lebesgue measure.

Moreover, the \emph{Hausdorff dimension} $\dim X$ of $X$ is defined by
$$
\dim X:=\inf\{s:\mathcal{H}^s(X)=0\}=\sup\{s:\mathcal{H}^s=\infty\}.
$$
Namely the Hausdorff dimension of $X$ is the unique point $s$ where the Hausdorff $s$-measure jumps from $\infty$ to $0$. It is particularly useful when analyzing the sizes of fractals. Though we do not attempt to make this precise, we will remark in passing that the set $A_n(\psi)$ in our question behaves in many aspects like a fractal, and in particular is self-similar.

We conclude this section by stating the \emph{Hausdorff-Cantelli Lemma}\footnote{This is referred as the Hausdorff-Cantelli lemma in \cite{BD} for convenience and should not be taken as a joint work.}, which will be used repeatedly in \S{\ref{s3}}. 
\begin{lem}[Hausdorff-Cantelli]\label{l0}
Let $H_i$ be a sequence of intervals in $\mathbb{R}$ and suppose that for some $s>0$,
$$
\sum_{i}(|H_i|)^s<\infty,
$$
then $\mathcal{H}^s(\limsup H_i)=0$.
\end{lem}
Here $|H_i|$ is the Lebesgue measure of $H_i$ and the limsup set
$$
\limsup H_i:=\bigcap_{i=1}^\infty\bigcup_{j=i}^\infty H_j
$$
contains exactly those elements that belong to infinitely many of $H_i$.

\section{Proof of the main theorem}\label{s3}

In this section, we embark on the proof of the Main Theorem. For convenience we make some standard simplifications.

By changing coordinates, a $C^2$ planar curve $\mathcal{C}$ is locally represented by graphs $\mathcal{C}_f:=\{(x,f(x)):x\in I\}$ of $C^2$ functions $f$ on some compact interval $I$. 

Clearly, the local coordinate map $\mathbf{f}: I\rightarrow\mathcal{C}_f$ which sends $x$ to $(x,f(x))$ is bi-Lipschitz and hence for any set $\mathcal{A}\in I$ we have $$\mathcal{H}^s(\mathcal{A})\asymp \mathcal{H}^s(\mathbf{f}(\mathcal{A})).$$

Let $B:=\{x\in I:f''(x)=0\}$. Recall that non-degeneracy at a point $x$ simply means that $f''(x)\not=0$. Since $\mathcal{C}$ is non-degenerate everywhere except on a set of zero Hausdorff $s$-measure, we have $\mathcal{H}^s(B)=0$. By continuity, $B$ is a closed set in $I$, hence a simple measure-theoretic argument allows us to write $I\setminus B$ as a countable union of bounded open intervals $I_i$ on which $f$ satisfies 
\begin{equation}\label{e20}
0<c_1\le|f''(x)|\le c_2<\infty.
\end{equation}
for all $x\in I_i$. Here the positive constants $c_1$ and $c_2$ depend on the particular choice of interval $I_i$. Since the Hausdorff measure is countably subadditive, to prove the Main Theorem for all $C^2$ non-degenerate planar curves, it suffices to prove it for all arcs $\mathcal{C}_f:=\{(x,f(x)):x\in I\}$ associated with some $C^2$ function $f$ whose second derivative is bounded and bounded away from zero. Thus, without loss of generality, we will assume $f$ satisfies \eqref{e20} on $I$.

Let $\psi(q):\mathbb{N}\rightarrow\mathbb{R}^{+}$ be an approximation function. Write $I:=[a,b]$. For $(q_1,q_2,p)\in\mathbb{Z}^3$, let
\begin{equation}\label{e1.1}
\mu(q_1,q_2,p)=\{x\in I:|q_1x+q_2f(x)-p|<\psi(q)\}
\end{equation}
where $q=\max\{|q_1|,|q_2|\}$. We are primarily interested in the measure theoretic properties of the set
\begin{equation}\label{e1.2}
\mu(\psi)=\{x\in I:\exists^\infty (q_1,q_2,p)\in\mathbb{Z}^3\text{ such that }x\in\mu(q_1,q_2,p)\}.
\end{equation}•

Notice that $\mathcal{H}^s(\mu(\psi))=0$ if and only if $\mathcal{H}^s(A_2(\psi)\cap\mathcal{C}_f)=0$. Thus the convergence case of the Main Theorem  reduces to

\begin{prop}\label{t1}
Let $s\in(0,1]$. Then
\[
\mathcal{H}^s(\mu(\psi))=
0\quad\text{if}\quad\sum_{q=1}^\infty \psi(q)^sq^{2-s}\text{ converges}.
\]
\end{prop}

First of all, due to the convexity assumption \eqref{e20}, we observe the fact that $\mu(q_1,q_2,p)$ is either an interval or the union of two intervals.
Hence in view of the Hausdorff-Cantelli Lemma, to prove the proposition it suffices to show the convergence of the series
\begin{equation}\label{e3.1}
\sum_{\substack{q_1,q_2,p\in\mathbb{Z}\\(q_1,q_2)\not=(0,0)}}|\mu(q_1,q_2,p)|^s
\end{equation} 
provided that 
\begin{equation}\label{e3.2}
\sum_{q=1}^\infty\psi(q)^sq^{2-s}<\infty.
\end{equation}

We may assume, by a standard reduction argument, 
\begin{equation}\label{e3.0}
\psi(q)\ge q^{1-\frac{3+\varepsilon_0}s}
\end{equation}
with some fixed small $\varepsilon_0>0$. Because if $\psi$ does not satisfy \eqref{e3.0}, we may take a new approximation function
$$\hat\psi(q):=\max\left\{\psi(q),q^{1-\frac{3+\varepsilon_0}s}\right\},$$
which satisfies \eqref{e3.0}.
One can easily check that 
$$\sum_{q=1}^\infty\hat\psi(q)^sq^{2-s}<\infty.$$
We may define the corresponding set $\hat\mu(q_1,q_2,p)$ as in \eqref{e1.1} with $\psi$ replaced by $\hat\psi$. Clearly the inclusion relation
$$\mu(q_1,q_2,p)\subseteq\hat\mu(q_1,q_2,p)$$
holds since $\hat\psi(q)\ge\psi(q)$. So if we can show $\mathcal{H}^s(\mu(\hat\psi))=0$, then $\mathcal{H}^s(\mu(\psi))=0$ follows.
Hence without loss of generality, for the rest of the argument, we will assume \eqref{e3.0}.

We will divide the set $\mathbb{Z}^2\backslash(0,0)$ of all possible choices for $(q_1,q_2)$ into a couple of subsets, which will be dealt with separately. We also notice that for given $(q_1,q_2)$, there are only finitely many $p\in\mathbb{Z}$ such that $\mu(q_1,q_2,p)\not=\emptyset$. Indeed, in view of the definition \eqref{e1.1} such $p$ must satisfy
\begin{equation}\label{e3.3}
|p|\le C q=C\max\{|q_1|,|q_2|\}
\end{equation}
where $$C=\max_{x\in I}\{|x|+|f(x)|+1\}.$$

Let
$$M=1+\max_{x\in I}|f'(x)|$$
and
$$\Theta_1=\{(q_1,q_2)\in \mathbb{Z}\times\mathbb{Z}:|q_1|>2M|q_2|\}$$
and
$$\Theta_2=\mathbb{Z}^2\setminus\Theta_1\cup(0,0).$$

We state a lemma (see \cite[Lemma 9.7]{Ha}) that will be used repeatedly in our argument. 

\begin{lem}\label{l2}
Let $h(x)\in C^2(I)$ be such that $\min_{x\in I}|h'(x)|=\delta_1$ and $\min_{x\in I}|h''(x)|=\delta_2$. For $\eta>0$, define
$$E(\eta):=\{\eta\in I:|h(x)|<\eta\}.$$ Then we have
$$|E(\eta)|\ll\min\left(\frac\eta{\delta_1},\sqrt{\frac{\eta}{\delta_2}}\,\right).$$
\end{lem}

We consider the case $(q_1,q_2)\in\Theta_1$ first. In this case, we have
$$|q_1+q_2f'(x)|\ge|q_1|-M|q_2|\ge\frac{|q_1|}2.$$
Now by Lemma \ref{l2}, we know
$$
|\mu(q_1,q_2,p)|\ll \frac{\psi(|q_1|)}{|q_1|}.
$$
Moreover, for given $q_1$, there are at most $\ll q_1^2$ possible choices for $q_2$ and $p$. 
Therefore the total contribution from the $(q_1,q_2)\in \Theta_1$ to \eqref{e3.1} is
\begin{align*}
&\sum_{\substack{(q_1,q_2)\in\Theta_1\\p\in\mathbb{Z}}}|\mu(q_1,q_2,p)|^s\\
{\ll}&\sum_{q_1\in \mathbb{Z}^*}(\psi(|q_1|)/|q_1|)^s|q_1|^2
\end{align*}
which is convergent by \eqref{e3.2}.

From now on, we treat the much more difficult case $(q_1,q_2)\in\Theta_2$. Clearly  in this case $q_2\not=0$ and moreover $|q_2|\asymp q$.
For convenience, we may extend the definition of $f(x)$ to $\mathbb{R}$  by taking the second order Taylor expansions at the end points of $I$. Namely
let 
$$f(x)=f(b)+f'(b)(x-b)+\frac{f''(b)}2(x-b)^2$$
when $x>b$
and
$$f(x)=f(a)+f'(a)(x-a)+\frac{f''(a)}2(x-a)^2$$
when $x<a$.
Clearly the extended $f$ satisfies $f\in C^2(\mathbb{R})$ and $c_1\le |f''|\le c_2$. Since $f''$ does not change sign throughout $\mathbb{R}$, $f'$ is strictly monotonic on $\mathbb{R}$ and has range $(-\infty,\infty)$. Let $g(y):\mathbb{R}\rightarrow\mathbb{R}$ be the inverse function of $-f'(x)$. To this end,
let $$x_0:=g(q_1/q_2)$$ 
which is the unique point $x_0\in\mathbb{R}$ such that
\begin{equation}\label{e3.4}
q_1+q_2f'(x_0)=0.
\end{equation}•
Let $J=[-2M,2M]$ and $I'=g(J)\supset I$. So $x_0\in I'$. 
Note that
$$g'(y)=\frac1{f''(g(y))}$$
and hence that
\begin{equation}\label{e3.12}
c_2^{-1}\le|g'(y)|\le c_1^{-1}
\end{equation}
for all $y\in\mathbb{R}$. Thus by the mean value theorem
\begin{equation}\label{e3.5}
|I'|\le c_1^{-1}|J|= 4c_1^{-1}M.
\end{equation}

Now, we let
$$F(x)=q_1x+q_2f(x)$$ with $q_1/q_2\in J$. Then

\begin{lem}\label{l1}
$$|F'(x)|\asymp |q_2(F(x)-F(x_0))|^{1/2}$$
where the $\asymp$ constants depend only on $c_1$, $c_2$.
\end{lem}

\begin{proof}
Applying Taylor's theorem of order 2 about $x=x_0$ for $F(x)$ and noting $F'(x_0)=0$ by \eqref{e3.4}, we obtain
\begin{equation}\label{e3.6}
F(x)=F(x_0)+\frac{F''(\xi_1)}2(x-x_0)^2
\end{equation}
for some $\xi_1$ between $x$ and $x_0$.

Moreover, by the mean value theorem,
\begin{equation}\label{e3.7}
F'(x)=F'(x)-F'(x_0)=F''(\xi_2)(x-x_0)
\end{equation}
for some $\xi_2$ between $x$ and $x_0$.

Hence merging $\eqref{e3.6}$ and $\eqref{e3.7}$ yields
$$F(x)-F(x_0)=\frac{F''(\xi_1)}2\left(\frac{F'(x)}{F''(\xi_2)}\right)^2.$$
The lemma immediately follows on noticing that
$$c_1|q_2|\le|F''(x)|\le c_2|q_2|.$$

\end{proof}

Now let $p_0$ be the unique integer such that
\begin{equation}\label{e3.8}
-\frac12<F(x_0)-p_0\le\frac12.
\end{equation}
If $p\not=p_0$, then for $x\in\mu(q_1,q_2,p)$
\begin{align*}
|F(x)-F(x_0)|&=|p-p_0+F(x)-p-F(x_0)+p_0|\\
&\ge|p-p_0|-|F(x)-p|-|F(x_0)-p_0|\\
&\overset{\eqref{e3.8}}{\ge}|p-p_0|-\psi(q)-1/2\\
&\ge\frac13|p-p_0|
\end{align*}
provided that 
\begin{equation}\label{e3.9}
\psi(q)\le \frac1{8}.
\end{equation} 
Since $\psi(q)$ decreases monotonically to 0 as $q\rightarrow\infty$, there exists $q_0\in\mathbb{N}$ such that \eqref{e3.9} holds when $q\ge q_0$.

By Lemma \ref{l1} and then Lemma \ref{l2}, we get, if $p\not=p_0$ and $q\ge q_0$ then
\begin{equation}\label{e10}
|\mu(q_1,q_2,p)|\ll \psi(q)(|q_2||p-p_0|)^{-1/2}.
\end{equation}
Therefore for fixed $(q_1,q_2)\in\Theta_2$ with $\max\{|q_1|,|q_2|\}\ge q_0$
\begin{align*}
\sum_{p\not=p_0}|\mu(q_1,q_2,p)|^s&\ll \sum_{p\not=p_0}\psi(q)^s(q|p-p_0|)^{-s/2}\\
&\overset{\eqref{e3.3}}{\ll}\psi(q)^sq^{-s/2}q^{1-s/2}\\
&\ll \psi(q)^sq^{1-s}.
\end{align*}
For fixed $q$, there are at most $\ll q$ choices for $q_1, q_2$. Hence
$$\sum_{\substack{(q_1,q_2)\in\Theta_2\\\max\{|q_1|,|q_2|\}\ge q_0\\p\not=p_0}}|\mu(q_1,q_2,p)|^s\ll \sum_{q}\psi(q)^sq^{2-s}$$
which converges.

Hitherto, we are left with the most difficult case $p=p_0$. For $x\in\mu(q_1,q_2,p_0)$, we have
\begin{align*}
|F(x)-F(x_0)|&=|F(x)-p_0+p_0-F(x_0)| \\
&\ge\|F(x_0)\|-\psi(q)\\
&\ge\frac12\|F(x_0)\|
\end{align*}
provided that
$$\|F(x_0)\|\ge2\psi(q).$$
By Lemma \ref{l1} 
$$|F'(x)|\gg (q\|F(x_0)\|)^{1/2}.$$
Then by Lemma \ref{l2}
\begin{equation}\label{e3.10}
|\mu(q_1,q_2,p_0)|\ll\frac{\psi(q)}{q^{1/2}}\|F(x_0)\|^{-1/2}.
\end{equation}
On the other hand, if $\|F(x_0)\|<2\psi(q)$, 
then
\begin{equation}\label{e3.11}
|\mu(q_1,q_2,p_0)|\overset{\text{Lem.}\ref{l2}}{\ll} \sqrt{\frac{\psi(q)}q}.
\end{equation}

We now define the dual curve $f^*(y)$ of $f(x)$, whose derivative is the inverse function of $-f'(x)$. Namely
$$f^*(y):=yg(y)+f(g(y)).$$
It is readily seen that
\begin{equation}\label{e3.13}
(f^*)'(y)=g(y)+yg'(y)+f'(g(y))g'(y)=g(y)
\end{equation}
and that
\begin{equation}\label{e3.14}
q_2f^*(q_1/q_2)=q_1g(q_1/q_2)+q_2f(g(q_1/q_2))=q_1x_0+q_2f(x_0)=F(x_0).
\end{equation}

The following lemma, which can be found as Lemma 2.2 in \cite{VV}, is essentially due to Huxley \cite{Hu}.

\begin{lem}[Huxley]\label{l3}
Suppose that $\phi$ has a continuous second derivative on a bounded interval $\Gamma$ which is bounded away from 0, and let $\delta\in(0,\frac14)$. Then for any $\varepsilon>0$ and $R\ge1$,
$$\sum_{R\le r<2R}\sum_{\substack{t/r\in\Gamma\\\|r\phi(t/r)\|<\delta}}1\ll_\varepsilon \delta^{1-\varepsilon}R^2+R\log(2R).$$ 
\end{lem}

Lemma \ref{l4} is a consequence of Lemma \ref{l3}.

\begin{lem}\label{l4}
Under the same conditions with Lemma \ref{l3}, for $\lambda\in(0,1)$ and $R\ge1$,
$$\sum_{R\le r<2R}\sum_{\substack{t/r\in\Gamma\\\|r\phi(t/r)\|\ge\delta}}\left\|r\phi\left(\frac{t}{r}\right)\right\|^{-\lambda}\ll R^2+\delta^{-\lambda}R\log(2R).$$ 
\end{lem}

\begin{proof}
We can restrict our attention to the terms with $\|r\phi(t/r)\|<1/4$ since the terms with $\|r\phi(t/r)\|\ge1/4$ clearly contribute $\ll R^2$ in total to the sum.
We separate the values of $\|r\phi(t/r)\|$ into dyadic ranges and obtain
\begin{align*}
\sum_{R\le r<2R}\sum_{\substack{t/r\in\Gamma\\1/4>\|r\phi(t/r)\|>\delta}}\left\|r\phi\left(\frac{t}{r}\right)\right\|^{-\lambda}
&\ll\sum_{j\le\log_2\frac1{4\delta}}\sum_{R\le r<2R}\sum_{\substack{t/r\in\Gamma\\\|r\phi(t/r)\|<2^j\delta}}(2^{j-1}\delta)^{-\lambda}\\
&\overset{\text{Lem.}\ref{l3}}{\ll}\sum_{j}2^{-j\lambda}\delta^{-\lambda}\left((2^j\delta)^{1-\varepsilon}R^2+R\log(2R)\right)\\
&\ll 2^{(1-\lambda-\varepsilon)\log_2\frac1{4\delta}}\delta^{1-\lambda-\varepsilon}R^2+\delta^{-\lambda}R\log(2R)\\
&\ll R^2+\delta^{-\lambda}R\log(2R).
\end{align*}
Here we choose $\varepsilon=\frac{1-\lambda}2$ when applying Lemma \ref{l3}.
\end{proof}

Finally, we are poised to prove that the series
\begin{equation}\label{e3.15}
\sum_{(q_1,q_2)\in\Theta_2}|\mu(q_1,q_2,p_0)|^s<\infty
\end{equation}
and hence conclude the proof of the proposition. We further divide this into two cases, namely $\|F(x_0)\|<2\psi(q)$ and $\|F(x_0)\|\ge2\psi(q)$.

For any integer $k$ with $2^k\ge q_0$,
\begin{align*}
&\sum_{\substack{(q_1,q_2)\in\Theta_2\\2^{k}\le |q_2|<2^{k+1}\\\|F(x_0)\|<2\psi(q)}}|\mu(q_1,q_2,p_0)|^s\\
\overset{\eqref{e3.14}\&\eqref{e3.11}}{\ll}& \sum_{2^{k}\le q_2<2^{k+1}}\sum_{\substack{q_1/q_2\in J\\\|q_2f^*(q_1/q_2)\|<2\psi(2^k)}}\left(\frac{\psi(2^k)}{2^k}\right)^{\frac{s}2}\\
\overset{\text{Lem.}\ref{l3}}{\ll}&\left(\psi(2^k)^{1-\varepsilon}2^{2k}+k2^k\right)\left(\frac{\psi(2^k)}{2^k}\right)^{\frac{s}2}\\
\ll&\psi(2^k)^{1+\frac{s}2-\varepsilon}2^{(2-s/2)k}+k\psi(2^k)^\frac{s}2 2^{(1-s/2)k}.
\end{align*}
This is
$$\ll \psi(2^k)^s(2^k)^{3-s}$$
provided that
\begin{equation}\label{e3.16}
\psi(2^k)\ge (2^k)^{1-\frac{4-\varepsilon_1}s}
\end{equation}
for some small $\varepsilon_1>0$. Now recalling \eqref{e3.0}, the condition \eqref{e3.16} is always satisfied! Therefore
\begin{equation}\label{e3.17}
\sum_{\substack{(q_1,q_2)\in\Theta_2\\|q_2|\ge q_0\\\|F(x_0)\|<2\psi(q)}}|\mu(q_1,q_2,p_0)|^s\ll \sum_k\psi(2^k)^s(2^k)^{3-s}\overset{\eqref{e3.2}}{<}\infty.
\end{equation}

The other case can be treated in a similar fashion. 
\begin{align*}
&\sum_{\substack{(q_1,q_2)\in\Theta_2\\2^{k}\le |q_2|<2^{k+1}\\\|F(x_0)\|\ge2\psi(q)}}|\mu(q_1,q_2,p_0)|^s\\
\overset{\eqref{e3.14}\&\eqref{e3.10}}{\ll}& \sum_{2^{k}\le q_2<2^{k+1}}\sum_{\substack{q_1/q_2\in J\\\|q_2f^*(q_1/q_2)\|\ge2\psi(2^k)}}\left(\frac{\psi(2^k)}{2^{k/2}}\right)^s\|q_2f^*(q_1/q_2)\|^{-s/2}\\
\overset{\text{Lem.}\ref{l4}}{\ll}&\left(\psi(2^k)^{-s/2}2^{k}k+2^{2k}\right)\left(\frac{\psi(2^k)}{2^{k/2}}\right)^s\\
\ll&\psi(2^k)^{\frac{s}2}(2^k)^{1-s/2}k+\psi(2^k)^s (2^k)^{(2-s/2)}.
\end{align*}
Again this is
$$\ll \psi(2^k)^s(2^k)^{3-s}$$
provided that \eqref{e3.16} holds, which is guaranteed by \eqref{e3.0}. So
\begin{equation}\label{e3.18}
\sum_{\substack{(q_1,q_2)\in\Theta_2\\|q_2|\ge q_0\\\|F(x_0)\|\ge2\psi(q)}}|\mu(q_1,q_2,p_0)|^s\ll \sum_k\psi(2^k)^s(2^k)^{3-s}\overset{\eqref{e3.2}}{<}\infty.
\end{equation}

Now, \eqref{e3.15} follows from \eqref{e3.17} and \eqref{e3.18}. Then the proposition follows from \eqref{e3.1} and the Hausdorff-Cantelli lemma.

\proof[Acknowledgments]
The author is very grateful to Prof. Friedlander for providing generous feedbacks to improve the presentation of this paper.

\end{document}